\documentclass{amsart}
\usepackage[top=1in,left=1in,right=1in,bottom=1in]{geometry}

\usepackage{ytableau,tikz}

\theoremstyle{plain}
\newtheorem{thm}{Theorem}[section]
\newtheorem{lem}[thm]{Lemma}
\newtheorem{cor}[thm]{Corollary}
\newtheorem{prop}[thm]{Proposition}

\theoremstyle{definition}
\newtheorem{defi}[thm]{Definition}
\newtheorem{ex}[thm]{Example}
\newtheorem{rmk}[thm]{Remark}
\newtheorem{alg}[thm]{Algorithm}
\newtheorem{ques}[thm]{Question}

\newcommand{\fkS}{\mathfrak{S}}
\newcommand{\fkG}{\mathfrak{G}}

\newcommand{\bfx}{\mathbf{x}}
\newcommand{\bfy}{\mathbf{y}}

\newcommand{\ZZ}{\mathbb{Z}}

\newcommand{\CC}{\mathbb{C}}

\newcommand{\wt}{\mathrm{wt}}
\newcommand{\SSYT}{\mathrm{SSYT}}

\newcommand{\RSDT}{\mathrm{SetRSDT}}
\newcommand{\supp}{\mathrm{Supp}}
\newcommand{\des}{\mathrm{Des}}
\newcommand{\SetSSYT}{\mathrm{SetSSYT}}

\newcommand{\cev}{\overleftarrow}

\title[On vexillary double Edelman--Greene coefficients]{Tableau formula for vexillary double Edelman--Greene coefficients}

\author{Adam Gregory, Zachary Hamaker \and Tianyi Yu}

\usepackage[backend=bibtex]{biblatex}
\begin{document}

\begin{abstract}
Lam, Lee and Shimozono recently introduced backstable double Grothendieck polynomials to represent $K$--theory classes of the infinite flag variety.
They used them to define double $\beta$--Stanley symmetric functions, which expand into double stable Grothendieck functions with polynomial coefficients called double $\beta$--Edelman--Greene coefficients.
Anderson proved these coefficients are $\beta$--Graham positive.
For vexillary permutations, this is equivalent to a statement for skew flagged double $\beta$--Grothendieck functions.
Working in this setting, we give a tableau formula for vexillary double $\beta$--Edelman--Greene coefficients that is manifestly $\beta$--Graham positive.
Our formula demonstrates a finer notion of positivity than was previously known.
\end{abstract}

\maketitle

\section{Introduction}

The goal of this paper to understand combinatorially certain geometric positivity results in cohomology and $K$--theory from~\cite{anderson2023strong}.
Specifically, Anderson shows double $\beta$--Stanley symmetric functions and double stable $\beta$--Grothendieck functions from~\cite{lam2021Ktheory} represent certain degeneracy loci, with the former expanding into the latter.
The coefficients are polynomials called \emph{double $\beta$--Edelman--Greene coefficients} and can include negative terms, but exhibit a finer notion of \emph{Graham positivity}, as was conjectured in~\cite{lam2021Ktheory} and proved in~\cite{anderson2023strong} using geometric methods.
We give a combinatorial description of these coefficients for vexillary permutations in terms of tableaux that is manifestly Graham positive.
Earlier combinatorial interpretations of double Edelman--Greene coefficients are not obviously Graham positive even in the cohomology case $\beta = 0$.
Classical Edelman--Greene coefficients from~\cite{edelman1987balanced} compute the Schur expansion of Stanley symmetric functions from~\cite{stanley1984number}.
They generalize Littlewood--Richardson coefficients and have been used to describe several more general families of Schubert structure coefficients~\cite{huang2023schubert}.
Similarly, while double $\beta$--Edelman--Greene coefficients are little studied, they are known to compute double $K$--homology structure coefficients in the Grassmannian.
As such, our work provides a new entry point towards the combinatorial understanding of Schubert calculus.

Let $w \in S_\ZZ$, the permutations of $\ZZ$ fixing all but finitely many values.
We make use of the variables $\bfx = (x_i)_{i \in \ZZ}$, $\bfy = (y_j)_{j \in \ZZ}$, $\bfx_+ = (x_i)_{i \geq 1}$, and $\bfx_- = (x_i)_{i \leq 0}$.
Each double Stanley symmetric function is a specialization of a \emph{backstable $\beta$--Grothendieck polynomial} $\cev{\fkG}_w(\beta;\bfx;\bfy)$, which represents both the connective $K$--theory for certain degeneracy loci~\cite{anderson2023strong} and, with $\beta = -1$, the $K$--theory of Schubert varieties in the infinite flag variety~\cite{lam2021Ktheory}.
Each $\cev{\fkG}_w(\beta;\bfx;\bfy)$ expands into \emph{double stable $\beta$--Grothendieck functions} $G_\lambda(\beta;\bfx_-;\bfy)$, which represent the connective $K$--theory of the Grassmannian:
\begin{equation}
    \label{eq:Kbackstable}
    \cev{\fkG}_w(\beta;\bfx;\bfy) = \sum_\mu a^{w}_\mu(\beta;\bfx;\bfy) \cdot G_\mu(\beta;\bfx_-;\bfy)
\end{equation}
where $a^{w}_\mu(\beta;\bfx;\bfy)$ is a rational function.
Mapping $\bfx \mapsto \bfy$ in these coefficients gives $j^{w}_\mu(\beta;\bfy) = a^{w}_\mu(\beta;\bfy;\bfy)$, and the \emph{double $\beta$--Stanley symmetric function} of $w$ is 
\begin{equation}
\label{eq:KStanley}
    F_w(\beta;\bfx;\bfy)  = \sum_\mu j^{w}_\mu(\beta;\bfy) \cdot G_\mu(\beta;\bfx_-;\bfy).
\end{equation}

Let $a \ominus b = \frac{a-b}{1+\beta b}$.
Say $f(\beta;\bfy)$ is \emph{$\beta$--Graham positive} if $f(\beta;\bfy) \in \ZZ_{\geq 0}[\beta(y_i \ominus y_j) : i \prec j]$ where $1 \prec 2 \prec \dots \prec -2 \prec -1 \prec 0$.
In~\cite{anderson2023strong}, Anderson showed $\beta^{\ell(w)-|\lambda|}j_\lambda^w(\beta;\bfy)$ is $\beta$--Graham positive, resolving~\cite[Conj 8.23]{lam2021Ktheory}.
Our main theorem is a refinement of Anderson's result for vexillary $w$.
A \emph{$\beta$--Graham positive monomial} is a product of terms $\beta(y_i \ominus y_j)$ with $i \prec j$.
Define three types of terms $y_i \ominus y_j$: 
\[
\mbox{Type 1:}\ \ 0 < i <j,\quad \mbox{Type 2:} \ \ i < j \leq 0, \quad \mbox{Type 3:} \ \ j \leq 0 < i. 
\]
\begin{thm}
    \label{t:Kmain}
    For $w \in S_\ZZ$ vexillary, $\beta^{\ell(w)-|\mu|}j^{w}_\mu(\beta;\bfy)$ is a sum of $\beta$--Graham positive monomials indexed by tableaux.
    Each such monomial contains $\beta(y_i \ominus y_j)$ at most twice if it is Type 3 and at most once otherwise.
    Furthermore, for all $\mu$ these monomials will only have terms of one of the Types 1 or 2.
\end{thm}

Our proof is combinatorial and implicitly gives a tableau formula for $j^w_{\mu}(\beta;\bfy)$.

The condition that $j^w_\mu$ can be expressed without using terms of both Type 1 and Type 2 does not appear in Anderson's work, and does not hold for $j^w_\mu(\beta;\bfy)$ when $w$ is not vexillary.
For example, with $\overline{k} = - k$ the non-vexillary permutation $w = s_2 s_1 s_{\overline{1}} s_0$ has
    \[
    j_{(1)}^w(\beta;\bfy) =  (y_1 \ominus y_2) (y_{\overline{1}} \ominus y_0) (y_1 \ominus y_0), 
    \]
    which involves terms of all three types.

As a corollary, setting $\beta = 0$ recovers the cohomology case.
Here $y_i \ominus y_j = y_i - y_j$. Define
\[
\cev{\fkS}_w(\bfx;\bfy) = \cev{\fkG}_w(0;\bfx;\bfy),  \quad s_\lambda(\bfx_-;\bfy) = G_\lambda(0;\bfx_-;\bfy), \quad j^w_\mu(\bfy) = j^{w}_\mu(0;\bfy).
\]
Then Theorem~\ref{t:Kmain} specializes to a  description of the \emph{double Edelman--Greene coefficient} $j^w_\mu(\bfy)$ from~\cite{lam2021back}, providing the first Graham positive combinatorial description of such coefficients as well.

Our approach to proving Theorem~\ref{t:Kmain} is based on tableau formulas for vexillary backstable double $\beta$--Grothendieck polynomials. 
By~\cite{wachs1985flagged}, 
each vexillary permutation $w$ is determined by its associated partition $\lambda(w)$ and flag $\phi(w)$.
For $w$ vexillary $\fkG_w(\beta;\bfx,\bfy) = G^{\phi(w)}_{\lambda(w)}(\beta;\bfx;\bfy)$, which is a \emph{flagged double stable $\beta$--Grothendieck function}.
Working in this setting, our proof of Theorem~\ref{t:Kmain} has three main steps.
We first establish the positivity of vexillary $j^w_\mu(\beta;\bfy)$ when $\phi(w)$ contains only non-negative integers.
Then, we use an extension $\tilde{\omega}$ of the $\omega$ involution for symmetric functions to cover the case where $\phi(w)$ has only non-positive integers.
Lastly, we combine and refine these previous arguments combine to prove Theorem~\ref{t:Krefined-ext}, a slight strengthening of Theorem~\ref{t:Kmain}.

\ 

\noindent \textbf{Organization:} In Section~\ref{s:background}, we give the necessary background to state Theorems~\ref{t:Kmain} precisely.
We prove special cases  of Theorem~\ref{t:Kmain} in Section~\ref{s:main}, then give a complete proof in Section~\ref{s:tableaux}.
We conclude with some final remarks in Section~\ref{s:remarks}.

\section{Preliminaries}
\label{s:background}
Write $\ZZ_{+} = \{i \in \ZZ : i \geq 1\}$, $\ZZ_{-} = \{j \in \ZZ : j \leq 0\}$ and $\ZZ_{\geq 0} = \ZZ_{+} \cup \{0\}$. 
Let $\mathcal{S}$ be the set of non-empty finite subsets of $\ZZ$, $\mathcal{S}_+$ the set of non-empty finite subsets of $\ZZ_+$, and $\mathcal{S}_{-}$ the set of non-empty finite subsets of $\ZZ_{-}$. 
Let $\overline{k} = -k$ and recall $\prec$ is the order on $\ZZ$ with $1 \prec 2 \prec \ldots \prec \overline{2} \prec \overline{1} \prec 0$.
Let $\bfx = (x_i)_{i \in \ZZ}$ and $\bfy = (y_j)_{j \in \ZZ}$ be commuting variables.
Write $\bfx_+ = (x_i)_{i \geq 1}$ and $\bfx_- = (x_i)_{i \leq 0}$.
Use $\bfx \mapsto \bfy$ to denote the substitution $x_i \mapsto y_i$ for $i \in \ZZ$, and $\bfx \leftrightarrow \bfy$ when swapping $x_j \leftrightarrow y_j$ for $j \in \ZZ$.

\subsection{Tableaux}
\label{ss:tabs}
A {\em partition} is a sequence $\lambda = (\lambda_1 \geq \ldots \geq \lambda_\ell)$ of weakly decreasing positive integers.
Here $\lambda$ has {\em size} $|\lambda| = \sum \lambda_i$ and {\em length} $\ell(\lambda) = \ell$, the number of nonzero {\em parts}.
Identify each partition $\lambda$ with its {\em Young diagram} 
\[
\{(r,c) : 1 \leq r \leq \ell(\lambda), \; \ 1 \leq c \leq \lambda_r\}
\]
whose elements are called {\em cells}, which we plot using matrix coordinates.
Viewing partitions as Young diagrams, for $\mu \subseteq \lambda$ the \emph{skew diagram} $\lambda / \mu$ is the set of cells in $\lambda$ that are not in $\mu$.

For $\mu \subseteq \lambda$, a $\mathcal{S}$-{\em tableau of shape} $\lambda / \mu$ is a filling $T: \lambda / \mu \rightarrow \mathcal{S}$ of the cells in $\lambda / \mu$ by integers.
Let $a \ominus b = \frac{a-b}{1+\beta b}$.
Each $\mathcal{S}$-tableau is assigned the monomial weight
\begin{equation}
\label{eq:tab-wt}
\wt(T) = \beta^{-|\lambda/\mu|}\prod_{(r,c) \in \lambda / \mu} \prod_{i \in T(r,c)} \beta (x_i \ominus y_{i + c -r}).
\end{equation}
A $\mathcal{S}$-tableau $T$ of shape $\lambda/\mu$ is {\em semistandard} if
\[
\max T(r,c) \leq \min T(r,c+1) \quad \mbox{and} \quad \max T(c,r) < \min T(r+1,c)
\]
whenever $(r,c),(r,c+1)$ or $(r,c),(r+1,c)$ are in $\lambda/\mu$.
This means $T$ increases strictly down columns and weakly across rows.
Write $\SetSSYT(\lambda / \mu)$ for the set of semistandard $\mathcal{S}$-tableaux of shape $\lambda / \mu$.
More generally, for $\mathcal{T}$ a set of tableaux, let $\mathcal{T}_+$ and $\mathcal{T}_-$ be those with entries in $\mathcal{S}_+$ and $\mathcal{S}_-$ respectively, e.g., $\SetSSYT_+(\lambda)$ is the set of semistandard $\mathcal{S}$--tableaux of shape $\lambda$ whose entries are sets of positive integers.

\begin{rmk}
\label{r:degree 1 in tableau}
For any $T \in \SetSSYT(\lambda/\mu)$, the weights of two elements are different.
\end{rmk}

\begin{ex}
\label{e:tableau}
The following $T$ is a semistandard $\mathcal{S}$-tableau of shape $(4,3,2)$:
\begin{equation}
\label{eq:wt}
T = \vcenter{\hbox{\begin{ytableau}
 \scriptstyle\overline{3},\overline{2} & \scriptstyle\overline{2},0 & \scriptstyle 0 & \scriptstyle 1\\
 \scriptstyle\overline{1} & \scriptstyle 1 & \scriptstyle 3\\
 \scriptstyle 0,2 & \scriptstyle 2
 \end{ytableau}}}
 \quad \mbox{has cell weights} \quad
 \vcenter{\hbox{
 \begin{tikzpicture}[scale=.78]
 	\node at (0,0) { $\scriptstyle \beta^2(x_{\overline{3}} \ominus y_{\overline{3}})(x_{\overline{2}} \ominus y_{\overline{2}})$};
 	\node at (4,0) { $\scriptstyle \beta^2(x_{\overline{2}} \ominus y_{\overline{1}})(x_0 \ominus y_{1})$};
  	\node at (7,0) { $\scriptstyle \beta(x_{0} \ominus y_{2})$};
   	\node at (9,0) { $\scriptstyle \beta(x_{1} \ominus y_{4})$};
   	
 	\node at (0,-1) { $\scriptstyle \beta(x_{\overline{1}} \ominus y_{\overline{2}})$};
 	\node at (4,-1) { $\scriptstyle \beta(x_1 \ominus y_{1})$};
  	\node at (7,-1) { $\scriptstyle \beta(x_{3} \ominus y_{4})$};
  	
 	\node at (0,-2) { $\scriptstyle \beta^2(x_{0} \ominus y_{\overline{2}})(x_{2} \ominus y_{0})$};
 	\node at (4,-2) { $\scriptstyle \beta(x_{2} \ominus y_{1})$};
 	
 	\draw (10,0.5) -- (-2,0.5) -- (-2, -2.5) -- (6, -2.5);
 	\draw (2,0.5) -- (2, -2.5);
   	\draw (6,0.5) -- (6, -2.5);
   	\draw (8,0.5) -- (8, -1.5);
   	\draw (10,0.5) -- (10, -0.5);
    \draw (-2,-0.5) -- (10,-0.5);
    \draw (-2,-1.5) -- (8,-1.5);

 \end{tikzpicture}}}.
\end{equation}
Then $\wt(T)$ is the product of $\beta^{-9}$ and the cell weights.
\end{ex}

We often specify the largest entry allowed for each row using a weakly increasing sequence $\phi = (\phi_1 \leq \ldots \leq \phi_\ell)$, called a {\em flag}. The \emph{$\phi$--flagged semistandard $\mathcal{S}$--tableaux} of shape $\lambda / \mu$ are
\[
\SetSSYT^\phi(\lambda / \mu) = \{T \in \SetSSYT(\lambda / \mu): \max T(r,c) \leq \phi_r \mbox{ for all } (r,c) \in \lambda / \mu \}.
\]
For instance, with $T$ as in~\eqref{eq:wt} we have $T \in \SetSSYT^{(1,3,3)}((4,3,2))$ and $T \notin \SetSSYT^{(1,2,3)}((4,3,2))$.

\subsection{Permutations}
\label{ss:perms}
Let $S_\ZZ$ denote the set of bijections $w: \ZZ \rightarrow \ZZ$ whose \emph{support} $\supp(w) = \{k \in \ZZ : w(k) \neq k\}$ is finite.
Write $S_\infty$ for the subset of $S_\ZZ$ with support contained in $\ZZ_{+}$.
Let $s_i$ denote the simple transposition $(i,i+1)$, and note that $S_\ZZ$ is generated by the $s_i$ for $i \in \ZZ$ with relations
\begin{equation}
s_i^2 = 1, \quad s_i s_j = s_j s_i\ \mbox{for}\ |i -j| > 1, \quad s_i s_{i+1} s_i = s_{i+1} s_i s_{i+1}\ \mbox{for} \ i \in \ZZ,    
\label{eq:braid-relations}
\end{equation}
the latter two called \emph{braid relations}.
Therefore, each $w \in S_\ZZ$ can be expressed as $w = s_{i_1} \dots s_{i_p}$.
Such an expression is called \emph{reduced} if it is of minimal length with associated \emph{reduced word} $(i_1,\dots,i_p)$. The value $i$ is a {\em descent} in $w \in S_\ZZ$ if $w(i) > w(i+1)$, and $\des(w) = \{i: w(i) > w(i+1)\}$ is the \emph{descent set} of $w$.
Say $w$ is $k$--{\em Grassmannian} if it has at most one descent in position $k$.
There is a bijection $\lambda \mapsto w_\lambda$ from partitions to $0$--Grassmannian permutations with $w_\lambda(i) = i + \lambda_{1-i}$ for $i \leq 0$ and $w(\lambda) = i - \lambda_i'$ for $i > 0$.

We make extensive use of two automorphisms of $S_\ZZ$: the map $\neg$ is defined by $\neg s_i = s_{\overline{i}}$ for $i \in \ZZ$, while $\iota$ is defined by $\iota(s_i) = s_{i+1}$ for $i \in \ZZ$.
The map $\neg$ is closely related to the better known automorphism conjugation by the long element.
Recall $S_n$ is the set of permutations with support in $\{1,2,\dots,n\}$.
For $w \in S_n$ and $w_0 = n \dots21$, recall we obtain $w_0 w w_0$ from $w$ by mapping $s_i$ to $s_{n-i}$ in any of its expressions.
Therefore $w_0 w w_0 = \iota^n \circ \neg(w)$.

Say $w \in S_\ZZ$ is {\em vexillary} if there are no $i < j < k < \ell$ with $w(j) < w(i) < w(\ell) < w(k)$.
Each vexillary permutation $w \in S_\ZZ$ determines a shape $\lambda(w)$ and flag $\phi(w)$ of the same length, as demonstrated by Wachs~\cite{wachs1985flagged}.
Let $I_k(w) = \{j > k : w(k) > w(j)\}$, $c_k(w) = |I_k(w)|$ and $p_k(w) = \min I_k(w)-1$ where $p_k(w)$ is not defined when $I_k(w)$ is empty.
The shape $\lambda(w)$ is obtained by sorting the \emph{code} $c(w) = (\dots, c_{\overline{1}}(w),c_0(w), c_1(w),\dots)$ in decreasing order, and the flag $\phi(w)$ is obtained by arranging the $p_k(w)$'s in increasing order.
We say flags $\phi$ and $\phi'$ are \emph{equivalent} with respect to the partition $\lambda$ if they are both compatible with $\lambda$ and $\SSYT^\phi(\lambda) = \SSYT^{\phi'}(\lambda)$.
For example, $w = 345162$ is vexillary with $\lambda(w) = (2,2,2,1)$ and $\phi(w) = (3,3,3,5)$.
Since $\SetSSYT^{(3,3,3,5)}((2,2,2,1)) = \SetSSYT^{(1,2,3,5)}((2,2,2,1))$, the flags $(3,3,3,5)$ and $(1,2,3,5)$ are equivalent with respect to $(2,2,2,1)$.
Other authors, e.g.~\cite{weigandt2021bumpless}, use the latter as the flag for $w$, but our choice of conventions is necessary for Lemma~\ref{l:neg-flag}.

The partition $\lambda(w)$ and flag $\phi(w)$ obtained from a vexillary $w \in S_\ZZ$ satisfy
\begin{equation}
\label{eq:compatible}
    \phi(w)_{i+1} - \phi(w)_i \leq \lambda(w)_i - \lambda(w)_{i+1} + 1 \quad \text{for each} \quad 1 \leq i \leq \ell(\lambda(w)).
\end{equation}
When this holds, we say the partition $\lambda$ and flag $\phi$ are {\em compatible}.
Let $w_{\lambda,\phi}$ denote the permutation determined by the subsequent result.

\begin{prop}[{\cite[(1.37) and (1.38)]{macdonald1991notes}}]
\label{p:comp-implies-vex}
The flag $\phi$ is compatible with the partition $\lambda$ if and only if there exists a vexillary permutation $w \in S_\ZZ$ so that $\lambda = \lambda(w)$ and $\phi$ is equivalent to $\phi(w)$ with respect to $\lambda$.
\end{prop}

The above result is implicit in work of Wachs~\cite{wachs1985flagged} and proved in Macdonald~\cite{macdonald1991notes} only for positive flags.
To see the result holds for $w \in S_\ZZ$ vexillary, note that $\iota(w)$ is vexillary and $\phi(\iota(w))$ is obtained from $\phi(w)$ by incrementing each entry by one.
Since this process is reversible, by applying $\iota$ sufficiently many times we can reduce to the case where $\phi$ has all positive entries.

Recall $\lambda'$ is the transpose of $\lambda$.
The following result translates~\cite[(1.42)]{macdonald1991notes} from conjugation by $w_0$ to $\neg$.

\begin{lem}
\label{l:neg-flag}
    The permutation $\neg w_{\lambda,\phi}$ is vexillary with shape $\lambda'$ and flag $\psi = \phi(\neg w_{\lambda,\phi})$, where $\psi_i = -\phi_{\lambda_i'}$.
\end{lem}

\begin{proof}
    Let $w \in S_n$ be vexillary with $\phi(w) = (f_1^{m_1},\dots,f_k^{m_k})$, so $f_i$ occurs $m_i$ times.
    Then~\cite[(1.42)]{macdonald1991notes} asserts $\lambda(w_0ww_0) = \lambda(w)'$ and $\phi(w_0ww_0) = (g_1^{n_1},\dots,g_k^{n_k})$ where $g_i = n-f_{k+1-i}$.
    Since $\neg(w) = \iota^{-n}(w_0ww_0)$, we see $\phi(\neg(w))$ has $i$th distinct term
    $g_i  - n = -f_{k+1-i}$.
    Since our conventions give $\phi_i = \phi_j$ when $\lambda_i = \lambda_j$, we see for $\psi_j = g_i$ that $\phi_{\lambda'_j} = f_{k+1-i}$.
\end{proof}

\subsection{Functions}
\label{ss:backstable-Schubert}

There are several equivalent conventions for defining ($\beta$--)Grothendieck polynomials.
Ours are equivalent to those from~\cite{lam2021Ktheory} after setting $\beta = -1$.
To begin, we define the {\em double stable $\beta$--Grothendieck} and {\em $\phi$-flagged double stable $\beta$--Grothendieck} functions as 
\begin{equation}
\label{eq:double-Schur-tab}
    G_\lambda(\beta;\bfx_{-};\bfy) = \sum_{T \in \SetSSYT_{-}(\lambda)} \wt(T) \quad \text{and} \quad G^{\phi}_\lambda(\beta;\bfx;\bfy) = \sum_{T \in \SetSSYT^\phi(\lambda)} \wt(T).
\end{equation}


To define backstable Grothendieck polynomials, we first define the usual double Grothendieck polynomial.
Each $\sigma \in S_\ZZ$ acts on the formal power series $f(\bfx,\bfy)$ by permuting the $\bfx$-variables, denoted $ f(\sigma \cdot \bfx;\bfy)$.
Let $\partial_i$ be the $i$th \emph{divided difference operator}, which acts on the polynomials by
\[
\partial_i(f(\bfx)) = \frac{f(\bfx) -  f(s_i \cdot \bfx)}{x_i - x_{i+1}},
\]
and define the $i$th \emph{$\beta$--isobaric divided difference operator} $\pi_i(f(\bfx)) = \partial_i((1+\beta x_i)f(\bfx))$.
The $\beta$--isobaric divided difference operators satisfy $\pi_i^2 = \beta \pi_i$ and the braid relations from~\eqref{eq:braid-relations}.
Therefore, we can define $\pi_w = \pi_{i_1} \dots \pi_{i_p}$ where $(i_1,\dots,i_p)$ is any reduced word for $w$.
The \emph{double $\beta$--Grothendieck polynomial} for a permutation $w \in S_n$ is 
\begin{equation}
    \label{eq:double-Schubert}
    \fkG_w(\beta;\bfx_+;\bfy_+) = \pi_{w^{-1}w_0} \prod_{\substack{1 \leq i,j \leq n-1,\\ i+j \leq n}} x_i \ominus y_j
\end{equation}
where $w_0 = n \dots 21$.
It is a remarkable consequence of this definition that $\fkG(\beta;\bfx_+;\bfy_+)$ is independent of $n$, hence we can view $w$ as an element of $S_\infty$.

Recall $\iota:S_\ZZ \to S_\ZZ$ where $\iota(s_i) = s_{i+1}$ for all $i \in \ZZ$, or equivalently $\iota(w)(i+1) = w(i)+1$.
Let $\gamma$ shift variables up by one: $\gamma(x_i) = x_{i+1}$ and $\gamma(y_i) = y_{i+1}$, and note $\gamma$ is invertible.
The \emph{backstable double $\beta$--Grothendieck polynomial} for $w \in S_\ZZ$ is
\[
\cev{\fkG}_w(\beta;\bfx;\bfy) = \lim_{p \to \infty} \gamma^{-p} \fkG_{\iota^p(w)}(\beta;\bfx_+;\bfy_+).
\]
Note this definition does not depend on $\fkG_w(\beta;\bfx;\bfy)$ when $w \notin S_\infty$ as for $p$ sufficiently large $\iota^p(w) \in S_\infty$.
By setting $\beta = 0$, we recover the \emph{double Schubert polynomial} $\fkS_w(\bfx;\bfy) = \fkG_w(0;\bfx;\bfy)$ and \emph{backstable double Schubert polynomial} $\cev{\fkS}(\bfx;\bfy) = \cev{\fkG}_w(0;\bfx;\bfy)$.

For $w$ vexillary, the following is a straightforward consequence of~\cite[Thm 5.8]{knutson2009geometry}.
\begin{prop}
    \label{p:gvex}
    For $w \in S_\ZZ$ vexillary, $\cev{\fkG}_w(\beta;\bfx;\bfy) = G^{\phi(w)}_{\lambda(w)}(\beta;\bfx;\bfy)$.
\end{prop}
\begin{proof}
By Proposition~\ref{p:comp-implies-vex} we know $\phi$ and $\lambda$ are compatible if and only if there exists $w \in S_\infty$ vexillary with $\lambda(w) = \lambda$ and $\phi(w) = \phi$.
For such $w$ the double Schubert polynomial equality $\fkG_w(\beta;\bfx_+;\bfy_+) = G_\lambda^\phi(\beta;\bfx_+;\bfy_+)$ is~\cite[Thm 5.8]{knutson2009geometry}.
Note $\phi(\iota(w))$ is obtained from $\phi(w)$ by incrementing each entry by one and $\gamma$ decrements every variable index by one.
    The result now follows from the definition of $\cev{\fkG}_w(\beta;\bfx;\bfy)$.
\end{proof}

An important special case is that of the $0$--Grassmannian permutations.

\begin{prop}[{\cite[Lem~5.32]{lam2021Ktheory}}]
\label{p:Kgrassmannian}
    For $\lambda$ a partition,
$
G_\lambda(\beta;\bfx_-;\bfy) = \cev{\fkG}_{w_\lambda}(\beta;\bfx;\bfy).
$
\end{prop}


%

For variables $\mathbf{z} = (z_i)_{i \in I}$, define
\[
R(\beta;\mathbf{z}) = \ZZ(\beta) \left[ z_i, \frac{1}{1+\beta z_i}: i \in I\right],
\]
so $\beta,\beta^{-1} \in R(\beta;\mathbf{z})$ independent of $\mathbf{z}$.
The double stable $\beta$--Grothendieck functions are closed under multiplication over $R(\beta;\bfy)$~\cite{pechenik2017equivariant} so they form a basis for an algebra, which we denote $\Gamma(\beta;\bfx_-;\bfy)$.

\begin{prop}
\label{p:backstable-grothendieck-ring}
For $w \in S_\ZZ$, $\cev{\fkG}_w(\beta;\bfx;\bfy) \in R(\beta;\bfx;\bfy) \otimes \Gamma(\beta;\bfx;\bfy)$.
\end{prop}
Proposition~\ref{p:backstable-grothendieck-ring} is arguably implicit in~\cite{lam2021Ktheory} and implies~\eqref{eq:Kbackstable}; we give a proof in Appendix~\ref{a:functions}.
As a consequence, we have
\[
\cev{\fkG}_w(\beta;\bfx;\bfy) = \sum_\mu a_\mu^{w}(\beta;\bfx;\bfy) \cdot G_\mu(\beta;\bfx_-;\bfy)
\]
with $a^w_\lambda(\beta;\bfx;\bfy) \in R(\beta;\bfx;\bfy)$.
By Proposition~\ref{p:comp-implies-vex}, for $\lambda$ and $\phi$ compatible this implies
\[
G^{\phi}_\lambda(\beta;\bfx;\bfy) = \sum_\mu a^{\lambda,\phi}_\mu(\beta;\bfx;\bfy) \cdot G_\mu(\beta;\bfx_-;\bfy),
\]
again with $a^{\lambda,\phi}_\mu(\beta;\bfx;\bfy) \in R(\beta;\bfx;\bfy)$.
The \emph{double $\beta$--Stanley symmetric function} of $w$ is
\[
F_w(\beta;\bfx_-;\bfy) = \sum_\lambda a_\lambda^{w}(\beta;\bfy;\bfy) \cdot G_\lambda(\beta;\bfx_-;\bfy).
\]
We call $j^w_\lambda(\beta;\bfy) := a_\lambda^{w}(\beta;\bfy;\bfy)$ the \emph{double $\beta$--Edelman--Greene coefficient}.
With $j^\lambda_\mu(\beta;\bfy) := a^\lambda_\mu(\beta;\bfx;\bfy) \mid_{\bfx \mapsto \bfy}$, for $w$ vexillary note 
\begin{equation}
\label{eq:EG-flagged}
j_\mu^{w}(\beta;\bfy) = j^{\lambda(w),\phi(w)}_\mu(\beta;\bfy).
\end{equation}

Define the map $\omega_1:R(\beta;\bfx;\bfy) \to R(\beta;\bfx;\bfy)$ by $\omega_1(x_i) = \ominus x_{1-i}$ and $\omega_1(y_i) = \ominus y_{1-i}$ and the map $\omega_2:\Gamma(\beta;\bfx_-;\bfy) \to \Gamma(\beta;\bfx_-;\bfy)$ by $\omega_2(G_\lambda(\beta;\bfx_-;\bfy)) = G_{\lambda'}(\beta;\bfx_-;\bfy)$.
Then $\tilde{\omega} = \omega_1\otimes \omega_2$ acts on backstable double $\beta$--Grothendieck functions.

\begin{lem}[{\cite[Prop 5.23]{lam2021Ktheory}}]
    \label{l:omega}
    For $w \in S_\ZZ$, 
\[
\tilde{\omega}\left(\cev{\fkG}_w(\beta;\bfx;\bfy)\right) = \cev{\fkG}_{\neg w}(\beta;\bfx;\bfy).
\]
\end{lem}

Combined with Lemma~\ref{l:neg-flag} and Proposition~\ref{p:gvex}, this implies:
\begin{cor}
\label{L: negation flag}
Consider compatible $\lambda$ and $\phi$.
Then $j^{\lambda, \phi}_\mu(\beta;\bfy) = \omega_1(j^{\lambda', \xi}_{\mu'}(\beta;\bfy))$,
where $\xi_i = - \phi_{\lambda'_i}$.
In particular, if $\phi$ has only non-positive
entries, 
then $\xi$ has only non-negative entries. 
\end{cor}
\section{Main Result}
\label{s:main}

\subsection{Edelman--Green coefficients via tableaux}

We start with a way to decompose a semistandard
$\mathcal{S}$--tableau into two tableaux.
\begin{defi}
Let $\lambda$ be a partition and take 
$T \in \SetSSYT(\lambda)$.
The non-positive values in cells of $T$ form
a tableau $T_-$ of shape $\nu \subseteq \lambda$ in $ \SetSSYT_-(\nu)$.
The positive values in cells of $T$ form a tableau $T_+$ of shape $\lambda/\mu$ in $\SSYT(\lambda/\mu)_+$.
\end{defi}
For example, the tableau
\[
	 T =
\begin{ytableau}
	\scriptstyle \overline{2},\overline{1} & \scriptstyle \overline{1}, 1 & \scriptstyle 1,3\\
	\scriptstyle 0, 1 & \scriptstyle 2 
\end{ytableau}
\quad \mbox{splits into} \quad T_- = \begin{ytableau}
\scriptstyle \overline{2}, \overline{1}	& \scriptstyle\overline{1} \\
\scriptstyle 0 
 \end{ytableau}
\quad \mbox{and} \quad
T_+ = \begin{ytableau}
 	\none & \scriptstyle 1 & \scriptstyle 1,3\\
 	\scriptstyle1 & \scriptstyle2
 \end{ytableau}\;.
\] 
We now describe the image of $\SetSSYT^\phi(\lambda)$
under this decomposition.
Say a skew shape $\lambda/\mu$ is \emph{disconnected} if it contains no adjacent cells.
For $\phi = (\phi_1,\dots,\phi_\ell)$ a flag, define $\phi_{-}$ by $(\phi_{-})_i = \min \{\phi_i, 0\}$. Note that $\phi_{-}$ is weakly increasing and has only non-positive numbers.

\begin{lem}
\label{l:K-decompose}
Let $\lambda$ be a partition and $\phi$ a flag.
For $T$ a set-valued tableau, the map $ T \mapsto (T_-,T_+)$ defined on
\begin{equation}
\label{eq:k-decomposition}
\SetSSYT^\phi(\lambda) \to \bigsqcup_{\nu \subseteq \lambda} \SetSSYT^{\phi^-}(\nu) \times \left( \bigsqcup_{\mu \subseteq \nu:\ \nu/\mu \ \mathrm{disconnected}} \SetSSYT^{\phi^+}(\lambda/\mu)\right),
\end{equation}
is a bijection.
\end{lem}

\begin{proof}
    Take $T \in \SetSSYT^\phi(\lambda)$.  Say $T_-$ and $T_+$
    have shapes $\nu$ and $\lambda/\mu$
    respectively. 
    We know $\mu$ (resp. $\nu$) consists of all $(r,c) \in \lambda$ so that $T(r,c) \subseteq \ZZ_-$ 
    (resp. $T(r,c) \cap \ZZ_- \neq \empty$).
    Thus, $\mu \subseteq \nu$.

    We clearly have
    $T_- \in \SetSSYT^{\phi^-}(\nu)$
    and $T_+ \in \SetSSYT^{\phi^+}(\lambda/\mu)$.
    To show $\nu/\mu$ is disconnected,
    we consider $(r,c) \in \nu/\mu$.
    Then $T(r,c)$ has both
    positive and non-positive entries. 
    Thus, should they exist,
    $T(r-1,c)$ and $T(r,c-1)$ (resp. $T(r+1,c)$ and $T(r,c+1)$)
    cannot have positive (resp. non-negative) entries,
    so they cannot live in $\nu/\mu$.

    To see our map is a bijection, take $T' \in \SetSSYT^{\phi^-}(\nu)$ and $T'' \in \SetSSYT^{\phi^+}(\lambda/\mu)$.
    Their cell-wise union $T = T' \cup T''$ is a tableau of shape $\lambda$. It is easy to check that $T\in \SetSSYT^\phi(\lambda)$, 
    $T_+ = T'$, $T_- = T''$.
    Both directions are clearly injective, completing our proof.
\end{proof}

We define the flag $\phi^+$
as $\phi^+_i = \max(\phi_i, 0)$.
Clearly, $\SetSSYT^{\phi^+}(\lambda/\mu) = \SetSSYT^{\phi^+}_+(\lambda/\mu)$.
Next we translate~\eqref{eq:k-decomposition}
from the language of tableaux to the language
of polynomials:
\begin{cor}
\label{c:flagged-G-tableaux}
Suppose the partition $\lambda$ and flag $\phi$ are compatible.
Then
\[
G^{\phi}_\lambda(\beta;\bfx; \bfy) = \sum_{\nu \subseteq \lambda}
G^{\phi^-}_\nu(\beta;\bfx; \bfy) 
\sum_{\mu \subseteq \nu:\ \nu/\mu\ \mathrm{disconnected}} \beta^{|\nu| - |\mu|}G^{\phi}_{\lambda/\mu}(\beta;\bfx_+; \bfy).
\]
\end{cor}
\begin{proof}
The left hand side is $\sum_T \wt(T)$
where $T \in \SetSSYT^\phi(\lambda)$.
On the right hand side, each valid choice of $\nu, \mu$
contributes $\sum_{T_1, T_2}\wt(T_1) \wt(T_2)$
where $T_1 \in \SetSSYT^{\phi^-}(\nu)$
and $T_2 \in \SetSSYT^{\phi^+}(\lambda/\mu)$.
Then this equation is 
implied by~\eqref{eq:k-decomposition}.
\end{proof}

We may restrict this corollary to 
$\phi$ with only positive entries,
obtaining a formula for $a^{\lambda, \phi}_\nu$.

\begin{cor}
\label{C: Decompose positive flag}
Suppose the partition $\lambda$ 
and the flag $\phi$
are compatible. 
Further assume $\phi$ has only non-negative entries. 
Then 
\[
a^{\lambda, \phi}_\nu(\beta; \bfx; \bfy)
= \sum_{\mu \subseteq \nu:\ \nu/\mu \ \mathrm{disconnected}} 
\beta^{|\nu| - |\mu|}G^\phi_{\lambda/\mu}(\beta; \bfx_+; \bfy).
\]
Consequently, $a^{\lambda, \phi}_\nu(\beta; \bfx; \bfy)$
is $0$ if $\lambda/\nu$ has a cell in row $i$
where $\phi_i = 0$.
\end{cor}
\begin{proof}
Since $\phi$ has only positive entries,
$\phi^-$ has only $0$s.
Therefore, in Corollary~\ref{c:flagged-G-tableaux}
the terms $G^{\phi^-}_\nu(\beta;\bfx; \bfy)$ 
become $G_\nu(\beta;\bfx; \bfy)$.
The equation in this corollary follows 
from the definition of $a^{\lambda, \phi}_\nu(\beta; \bfx; \bfy)$.

Now suppose $\lambda/\nu$ has a cell in row $i$
where $\phi_i = 0$.
Then $\lambda/\mu$ also has a cell in row $i$
for any $\mu \subseteq \nu$.
Thus $G^\phi_{\lambda/\mu}(\beta; \bfx_+; \bfy)$
vanishes,
so $a^{\lambda, \phi}_\nu(\beta; \bfx; \bfy) = 0$.
\end{proof}

Finally, for general compatible $\lambda$
and $\phi$,
we obtain the following tableau formula
involving a summation.
We define the flag $\phi^+$
as $\phi^+_i = \max(\phi_i, 0)$.

\begin{prop}
\label{P: decompose 2}
Suppose the partition $\lambda$ 
and the flag $\phi$
are compatible. 
Then
\begin{equation}
\label{eq: general G}
G^{\phi}_{\lambda}(\beta;\bfx; \bfy) 
= \sum_{\nu \subseteq \lambda}
G^{\phi^-}_{\nu}(\beta;\bfx; \bfy) a^{\lambda, \phi^+}_\nu(\beta;\bfx; \bfy).
\end{equation}
Consequently, 
\begin{equation}
\label{eq: general a}
a^{\lambda, \phi}_\rho(\beta; \bfx; \bfy)
= \sum_{\nu \subseteq \lambda}
a^{\nu, \phi^-}_{\rho}(\beta;\bfx; \bfy)
a^{\lambda, \phi^+}_\nu(\beta;\bfx; \bfy). 
\end{equation}
By setting $\bfx$ into $\bfy$,
we obtain:
\begin{equation}
\label{eq: general j}
j^{\lambda, \phi}_\rho(\beta; \bfy)
= \sum_{\nu \subseteq \lambda}
j^{\nu, \phi^-}_{\rho}(\beta; \bfy)
j^{\lambda, \phi^+}_\nu(\beta; \bfy). 
\end{equation}
\end{prop}
\begin{proof}
We first observe in Corollary~\ref{c:flagged-G-tableaux}
that we can replace $G^\phi_{\lambda/\mu}(\beta; \bfx_+; \bfy)$ with $G^{\phi^+}_{\lambda/\mu}(\beta; \bfx_+; \bfy)$:
Note $\phi$ and $\phi^+$
differ on entry $i$ only when $\phi_i < 0$.
If $\lambda/\mu$ has a cell in such row $i$,
the two polynomials both vanish. 
Otherwise, the two polynomials clearly agree. 

Next, since $\phi^+$ has only non-negative
entries, 
Corollary~\ref{C: Decompose positive flag} and Corollary~\ref{c:flagged-G-tableaux} give~\eqref{eq: general G}.

For~\eqref{eq: general a}, we first need to show
$\nu$ and $\phi^-$ are compatible to ensure the notation $a^{\nu, \phi^-}_\mu(\beta;\bfx; \bfy)$ makes sense. 
We check $\phi^-_{i+1} - \phi^-_i \leq \nu_{i} - \nu_{i+1} + 1$
by considering three cases.
\begin{itemize}
\item If $\phi_{i+1}, \phi_i \leq 0$,
then $\phi^-_{i+1} - \phi^-_{i} = \phi_{i+1} - \phi_i$.
By the condition of $\nu$,
$\nu_{i} - \nu_{i+1} = \lambda_{i} - \lambda_{i+1}$. 
Then the equation to check follows from
the compatibility of $\lambda$ and $\phi$.
\item If $\phi_i \leq 0 < \phi_{i+1}$,
then $\phi^-_{i} = \phi_i$ and $\phi^-_{i+1}  = 0 < \phi_{i+1}$.
By the condition on $\nu$,
we know $\nu_{i} = \lambda_i$ and $\nu_{i+1} < \lambda_{i+1}$.
We have 
$$\phi^-_{i+1} - \phi^-_i 
< \phi_{i+1} - \phi_i
\leq \lambda_{i} - \lambda_{i+1} + 1
< \nu_{i} - \nu_{i+1} + 1.$$
\item If $0 < \phi_i, \phi_{i+1}$,
we have $\phi^-_{i+1} - \phi^-_i = 0 - 0= 0$.
Our inequality is trivial. 
\end{itemize} 

Then~\eqref{eq: general a} follows by expanding each
$G^{\phi^-}_\nu(\beta;\bfx; \bfy)$ from~\eqref{eq: general G} into $G_{\rho}(\beta;\bfx_-; \bfy)$'s.
Finally, we obtain~\eqref{eq: general j}
by setting $\bfx$ to $\bfy$ 
in~\eqref{eq: general a}.
\end{proof}

Recall our main objective is to give a tableau formula for
$j^{\lambda, \phi}_\rho(\beta; \bfy)$ when $\lambda$ and $\phi$
are compatible.
By~\eqref{eq: general j},
this reduces to identifying formulas for 
$j^{\lambda, \phi^-}_\nu(\beta; \bfy)$
and $j^{\lambda, \phi^+}_\nu(\beta; \bfy)$. 
For the latter,
by Corollary~\ref{C: Decompose positive flag},
\begin{equation}
\label{eq:j-pos}
j^{\lambda, \phi^+}_\nu(\beta; \bfy)
= \sum_{\mu \subseteq \nu:\ \nu/\mu \ \mathrm{disconnected}} \beta^{|\nu| - |\mu|} G^{\phi^+}_{\lambda/\mu}(\beta; \bfx_+; \bfy) \mid_{\bfx \mapsto \bfy}.    
\end{equation}
Therefore,
we take a detour in the next
section and study 
$G^{\phi^+}_{\lambda/\mu}(\beta; \bfx_+; \bfy) \mid_{\bfx \mapsto \bfy}$.

\subsection{Understanding $G^\phi_{\lambda/\mu}(\beta; \bfx_+; \bfy) \mid_{\bfx \mapsto \bfy}$ when $\phi$ is non-negative}
\label{ss:pos-supp}
We fix $\lambda, \phi$ compatible and require
$\phi$ non-negative.
Also, fix $\mu \subseteq \lambda$.
Clearly, $G^\phi_{\lambda/\mu}(\beta; \bfx_+; \bfy) = 0$
if $\lambda/\mu$ has a cell in row $i$ and $\phi_i = 0$,
so we assume this is not the case. 
If $\lambda/\mu$ contains a cell on the diagonal, any tableau $T$ of shape $\lambda/\mu$ has a cell weight containing a term of the form $(x_i \ominus y_i)$ for some $i$.
Consequently, $G^{\phi}_{\lambda/\mu}(\beta;\bfx_+, \bfy)\mid_{\bfx \mapsto \bfy}$
must vanish. 
Therefore, we may assume $\lambda/\mu$
has no cells on the diagonal. 
Such skew shapes $\lambda/\mu$ can be separated into
two parts by the diagonal.
There exist partitions $\mu^U$ and $\mu^D$ satisfying:
\begin{itemize}
\item Cells of $\lambda/\mu^U$
and cells of $\lambda/\mu^D$ 
form a partition of cells in  $\lambda/\mu$.
\item Cells in $\lambda/\mu^U$
lies above the diagonal while cells in $\lambda/\mu^D$
lies below the diagonal.
\end{itemize}
Consequently, 
$G_{\lambda/\mu}^{\phi}(\beta;\bfx_+, \bfy)
\: = \: G_{\lambda/\mu^U}^{\phi}(\beta;\bfx_+, \bfy) \:
\cdot \: G_{\lambda/\mu^D}^{\phi}(\beta;\bfx_+, \bfy)$.

For instance, let $\lambda = (2,1,1)$, $\phi = (2,2,3)$
and $\mu = (1)$.
Then $\mu^U = (1,1,1)$ and $\mu^D = (2)$.
We have 
\[
G_{(2,1,1)/(1)}^{\phi}(\beta; \bfx_+, \bfy)
= G_{(2,1,1)/(1,1,1)}^{\phi}(\beta;\bfx_+, \bfy)G_{(2,1,1)/(2)}^{\phi}(\beta;\bfx_+, \bfy).
\]

We now must study the polynomials $G^\phi_{\lambda/\mu^U}(\beta;\bfx_+;\bfy)\mid_{\bfx \mapsto \bfy}$
and $G^\phi_{\lambda/\mu^D}(\beta;\bfx_+;\bfy)\mid_{\bfx \mapsto \bfy}$.
The former is relatively straightforward.

\begin{lem}
\label{l:upper diagonal pos}
We have
\[
G_{\lambda/\mu^U}^{\phi}(\beta;\bfx_+;\bfy)\mid_{\bfx \mapsto \bfy} = \sum_{T \in \SetSSYT^{\phi^+}(\lambda/\mu^U)} \wt(T)\mid_{\bfx \mapsto \bfy}.
\]
Moreover, each term $\wt(T)\mid_{\bfx \mapsto \bfy}$ is a product of distinct Type $1$ terms and a power of $\beta$. 
\end{lem}
\begin{proof}
Let $T \in \SetSSYT^{\phi}_+(\lambda/\mu)$.
We only need to show $\wt(T)\mid_{\bfx \mapsto \bfy}$
can be written as a product
of distinct $(y_i \ominus y_j)$
with $0 < i < j$ and a power of $\beta$.
A cell $(r,c)$ in $T$ satisfies 
$r < c$ and has cell weight
$x_{T(r,c)} \ominus y_{T(r,c) + c - r}$.
We have $0 < T(r,c) < T(r,c) + c - r$.
Thus, after setting $\bfx \mapsto \bfy$,
it becomes $(y_i \ominus y_j)$
with $0 < i < j$.
Then the proof is finished since 
other cells cannot have the same cell weight
as observed in Remark~\ref{r:degree 1 in tableau}.
\end{proof}

Our analysis of polynomial $G^\phi_{\lambda/\mu^D}(\beta;\bfx_+;\bfy)\mid_{\bfx \mapsto \bfy}$
is more subtle, requiring us to introduce a new flag.
For $\lambda$ and $\phi$ compatible, define $\psi(\lambda,\phi)$ as the sequence with
$\psi(\lambda,\phi)_i = \min(i - \lambda_i, \phi_i)$.
For the remainder of the section, $\lambda$ and $\phi$ will be clear, so abusing notation we write $\psi$.
Note $\psi$ is weakly increasing since $\phi_i \leq \phi_{i+1}$
and $i - \lambda_i < (i+1) - \lambda_{i+1}$.

\begin{lem}
\label{l:lower diagonal}
There exists a permutation $\pi\in S_\infty$ depending on 
$\lambda$ and $\phi$
so that
\begin{equation}
\label{EQ: pi lower diagonal}
G^{\phi}_{\lambda/\mu^D}(\beta;\bfx_+;\bfy)\mid_{\bfx \mapsto \bfy} = G^{\psi}_{\lambda/\mu^D}(\beta;\pi \cdot \bfx_+;\bfy) \mid_{\bfx \mapsto \bfy}.
\end{equation}
Here $\pi$ permutes the subscripts of the $\bfx$-variables. 
\end{lem}

We postpone the proof of Lemma~\ref{l:lower diagonal}. 
Let us first use it
to obtain the analogue of Lemma~\ref{l:upper diagonal pos}.

\begin{lem}
\label{l:lower diagonal2}
We can write
$G^{\phi}_{\lambda/\mu^D}(\beta;\bfx_+;\bfy)\mid_{\bfx \mapsto \bfy}$,
as a sum where each summand 
is a product of distinct Type $3$ terms and a power of $\beta$. 
\end{lem}
\begin{proof}
By Lemma~\ref{l:lower diagonal},
we must write the right hand side of~\eqref{EQ: pi lower diagonal} as such a sum. 
Let $T \in \SetSSYT_+^{\psi}(\lambda/\mu)$.
We first show $\wt(T)$
is a product of distinct $(x_i \ominus y_j)$
with $j \leq 0 < i$ and a power of $\beta$.
Consider a value $i \in T(r,c)$ with $(r,c) \in \lambda/\mu$.
We know $1 \leq i \leq \psi_r \leq r - \lambda_{r}$.
Thus, 
$$
T(r,c) + c - r \leq (r - \lambda_{r}) + c - r = c - \lambda_r \leq 0.
$$
Then this value contributes the weight $y_{\pi(i)} \ominus y_j$ to
the right hand side of~\eqref{EQ: pi lower diagonal} with $j \leq 0 < \pi(i)$,
which is a type $3$ term.
The proof is finished since 
cells cannot have the same cell weight, as observed in Remark~\ref{r:degree 1 in tableau}.
\end{proof}

We briefly summarize our conclusion in this subsection. 
\begin{rmk}
\label{R: Pos summary}
Consider $\lambda, \phi$ compatible.
Assume entries in $\phi$ are non-negative.
Take $\mu \subseteq \lambda$.
We can characterize $G =
G^{\phi}_{\lambda/\mu}(\beta;\bfx_+;\bfy)\mid_{\bfx \mapsto \bfy}$
as follows, ignoring the $\beta$ for clarity:
\begin{itemize}
\item (Case Zero): If $\lambda/\mu$ has a cell on the diagonal,
or in row $i$ with $\phi_i = 0$,
then $G = 0$.
\item (Case Up): If $\lambda/\mu$ 
is completely above the diagonal, 
$G$ is a sum where each summand is a distinct 
product of Type $1$ terms. 
\item (Case Down): If $\lambda/\mu$ 
is completely under the diagonal, 
$G$ is a sum where each summand is a distinct 
product of Type $3$ terms. 
\item (Case Combined): Otherwise, 
$G$ is a sum with each summand a distinct 
product of Type $1$ and Type $3$ terms.
\end{itemize}
\end{rmk}

In the remainder of this subsection, we will prove Lemma~\ref{l:lower diagonal}.
We rename $\mu^D$ as $\mu$ for simplicity. 
The first ingredient is a result showing $G_{\lambda / \mu}^{\phi}(\beta;\bfx_+;\bfy)$ is symmetric in certain $\bfx$-variables.

\begin{lem}
\label{l:flag-symm}
The polynomial $G_{\lambda/\mu}^{\phi}(\beta;\bfx_+;\bfy)$
is symmetric in $x_i, x_{i+1}, \cdots, x_{j}$
as long as $i, i+1, \cdots, j-1$ do not appear in $\phi$.
\end{lem}
\begin{proof}
We use a Bender-Knuth involution argument.
It is enough to prove the lemma 
with $j = i+1$.

Suppose $\lambda/\mu$ is a single row of $n$ cells. 
Then 
\[
G^{\phi}_{\lambda / \mu} (\beta;x_i, x_{i+1};\bfy) = G_{(n)}(\beta;x_i, x_{i+1}; \sigma \cdot \bfy),
\]
where $\sigma$ is some permutation shifting the $\bfy$-variables. 
This polynomial is well known to be symmetric in $x_i$ and $x_{i+1}$.
One argument for this fact is that, by Proposition~\ref{p:Kgrassmannian} $\cev{\fkG}_{\iota^{i+1}(w((n))}(\beta;\bfx;\bfy)$ is fixed by $\pi_i$ (up to a factor of $\beta$), which implies it is fixed by $s_i$ acting on the $\bfx$--variables.

Now consider an arbitrary $\lambda / \mu$.
Given $T \in \SetSSYT_+^\phi (\lambda / \mu)$, freeze all values in $T$ that are not $i$ or $i{+}1$.
Moreover, freeze all pairs $i,i{+}1$ with $i$ directly above $i{+}1$ in $T$ -- necessarily $i$ is the maximum value of its cell while $i{+}1$ is the minimum.
All other entries are free.
The free entries in $T$ will be in consecutive row segments, each with some weakly increasing filling from $\{i,i{+}1\}$.
Hence we obtain a partition of $\text{SSYT}_+^\phi (\lambda / \mu)$ into $A_1 \sqcup A_2 \sqcup \ldots$, where each $A_k$ consists of tableaux having the same frozen entries.
The contribution of a frozen pair of $i$ and $i+1$ to a tableau will be $(x_i \ominus y_j)(x_{i+1} \ominus y_j)$, which is symmetric in $x_i$ and $x_{i+1}$.
Consider each $g_k = \sum_{T \in A_k} \text{wt}(T)$.
The contribution of each non-frozen row segment is $G_{\lambda'/\mu'}(\beta;x_i, x_{i{+}1};\bfy)$ for some $\lambda'/\mu'$ that is a single row.
This is symmetric in $x_i$ and $x_{i+1}$ by the arguments above. 
\end{proof}

Recall $\psi$ is the flag with $\psi_i = \min(i-\lambda_i,\phi_i)$.
Let $\ell = \ell(\lambda)$.
Then let $\Delta$ be the sequence of length $\ell$ with $\Delta_i = \phi_i - \psi_i$.
\begin{lem}
\label{l:delta}
The sequence $\Delta$ is a weakly decreasing sequence of non-negative numbers. 
\end{lem}
\begin{proof}
First, notice that
\[
\Delta_i = \phi_i - \psi_i
= \phi_i - \min(\phi_i, i - \lambda_i)
= \max(0, \phi_i - i + \lambda_i).
\]
Thus,  $\Delta_i \geq 0$.
To see $\Delta$ is weakly decreasing,  we only need to show 
\[
\phi_i - i + \lambda_i \geq \phi_{i+1} - (i+1) + \lambda_{i+1}.
\]
After rearranging, this becomes:
\[
\lambda_i - \lambda_{i+1} + 1 \geq \phi_{i+1} - \phi_i,
\]
which follows from the compatibility of $\lambda$ and $\phi$.
\end{proof}

Next, we provide an algorithm to compute $\pi$ using $\lambda$ and $\phi$ as applied in Lemma~\ref{l:lower diagonal}.
Our algorithm is iterative, 
producing a sequence of
permutations $\pi_0, \ldots, \pi_\ell$ with $\pi = \pi_\ell$.
While defining these permutations, 
we make sure that
the numbers $1, \ldots, \Delta_i$ appear 
consecutively in increasing order in the one-line notation of $\pi_i$. 

\begin{alg}
Let $\pi_0$ be the identity permutation.
For $i \in [\ell]$:
\begin{itemize}
    \item  if $\psi_i \leq 0$,
set $\pi_{i} := \pi_{i-1}$;
\item if $\psi_i > 0$, construct $\pi_i$ from $\pi_{i-1}$ by shifting 
the values $1, \ldots, \Delta_i$
to the right until they are at 
indices $\psi_i+1, \ldots, \phi_i$.
\end{itemize}
\end{alg}

\begin{ex}
For $\lambda= (4,4,4,4,4,2,1)$ and $\phi = (3,4,4,5,6,6,6)$,
we have $\psi=(-3,-2,-1,0,1,4,6)$ and $\Delta = (6,6,5,5,5,2,2)$; so $\pi_0 = \pi_1 = \ldots = \pi_4 = \text{identity}$ and 
\begin{align*}
\pi_5 & = 6\underline{12345}78 \\
\pi_6 & = 6345\underline{12}78 \\
\pi_7 & = 634578\underline{12}.
\end{align*}
Here, we have underlined the consecutive sequences that have moved.
\end{ex}

Correspondingly, 
we define a sequence of flags.
For $i = 0, \ldots, \ell$,
let 
\[
\chi^{(i)} = (\psi_1,\ldots,\psi_i,\phi_{i+1},\ldots,\phi_{\ell}).
\]
Each $\chi^{(i)}$ is a flag since 
$\phi, \psi$ are both flags and
$\psi_i \leq \psi_{i+1} \leq \phi_{i+1}$.
As demonstrated by the following example, these $\chi^{(i)}$ characterizes the tableaux whose weights do not vanish after applying the modification $\pi_i \cdot \bfx$ and substitution $\bfx \mapsto \bfy$.

We summarize the role of the $\chi^{(i)}$ flags in the following lemma.

\begin{lem}
\label{l:all i the same}
For all $0 \leq i \leq \ell$, each 
$G_{\lambda/\mu}^{\chi^{(i)}}(\beta;\pi_i \cdot \bfx_+; \bfy) \mid_{\bfx \mapsto \bfy}$
is the same.
\end{lem}

\begin{proof}
Take $0 \leq i < \ell$.
We need to show 
\begin{equation}
\label{eq:lower diagonal}
G_{\lambda/\mu}^{\chi^{(i)}}(\beta;\pi_{i} \cdot \bfx_+; \bfy) \mid_{\bfx \mapsto \bfy}
= G_{\lambda/\mu}^{\chi^{(i+1)}}(\beta;\pi_{i+1} \cdot \bfx_+; \bfy)\mid_{\bfx \mapsto \bfy}.    
\end{equation}
This is separated into two steps.

\ 

\noindent \textbf{Step 1:} We first show 
\[
G_{\lambda/\mu}^{\chi^{(i)}}(\beta;\pi_{i} \cdot \bfx_+; \bfy)
= G_{\lambda/\mu}^{\chi^{(i)}}(\beta;\pi_{i+1} \cdot \bfx_+; \bfy).
\]
This is immediate if $\pi_i = \pi_{i+1}$.
Otherwise, 
$\pi_{i+1}$ is obtained from $\pi_i$
by moving $1, \cdots, \Delta_{i+1}$
to the right. 
In $\pi_i$, the number $1$ 
is at index $\max(\psi_i + 1, 1)$.
In $\pi_{i+1}$,
the number $\Delta_{i+1}$
is at index $\phi_{i+1}$.
As a summary,
$\pi_{i+1}$ is obtained from 
$\pi_i$ by permuting the numbers
whose indices are between
$\max(\psi_i + 1, 1)$ and $\phi_{i+1}$ inclusively. 

On the other hand, by Lemma~\ref{l:flag-symm},
$G_{\lambda/\mu}^{\chi^{(i)}}(\beta;\bfx_+; \bfy)$
is symmetric in variables 
with subscript between
$\max(\psi_i + 1, 1)$ and $\phi_{i+1}$ inclusively, so Step 1 follows. 

\ 

\noindent \textbf{Step 2:} We then show 
$$
G_{\lambda/\mu}^{\chi^{(i)}}(\beta;\pi_{i+1} \cdot \bfx_+; \bfy)\mid_{\bfx \mapsto \bfy}
= G_{\lambda/\mu}^{\chi^{(i+1)}}(\beta;\pi_{i+1} \cdot \bfx_+; \bfy)\mid_{\bfx \mapsto \bfy}.
$$
This is trivial when $\psi_{i+1} = \phi_{i+1}$,
which would imply $\chi^{(i)} = \chi^{(i+1)}$.
Suppose this is not the case, 
then $\psi_{i+1} = i - \lambda_{i+1}$.

Take $T \in \SetSSYT^{\chi^{(i)}}(\lambda/\mu)_+$
but not in 
$\SetSSYT^{\chi^{(i+1)}}(\lambda/\mu)_+$.
It is enough to show $\wt(T)$ becomes $0$
after we permute its $\bfx$-variables by $\pi_{i+1}$
and then set $\bfx \mapsto \bfy$.
Consider the cell 
$(i+1, \lambda_{i+1})$.
By the assumption of $T$,
this is a cell in $T$ containing a value $k$ with 
$\psi_{i+1} < k \leq \phi_i$.
Notice that the permutation $\pi_{i+1}$ sends
$j$ to $j - \psi_{i+1} = j - (i+1) + \lambda_{i+1}$
if $\psi_{i+1} < j \leq \phi_i$.
Thus, after we let $\pi_{i+1}$ permute
the $x$-variables, 
the value $k$ in this cell has weight
$$
x_{k- (i+1) + \lambda_{i+1}}
\ominus y_{k- (i+1) + \lambda_{i+1}},
$$
which would be $0$ after setting $\bfx \mapsto \bfy$.

To conclude, observe that~\eqref{eq:lower diagonal} follows from 
the two steps above. 
\end{proof}

We now prove Lemma~\ref{l:lower diagonal}.
\begin{proof}[Proof of Lemma~\ref{l:lower diagonal}]
Set $\pi = \pi_\ell$.
The statement follows from 
Lemma~\ref{l:all i the same}.
\end{proof}

We can now recover Anderson's positivity result for vexillary permutations.
By Proposition~\ref{P: decompose 2}, the result follows from the $\beta$--Graham positivity of $j^{\lambda,\phi}_\mu(\beta;\bfy)$ when $\phi$ has either all non-negative or all non-positive entries.
The former follows by combining Lemmas~\ref{l:upper diagonal pos} and~\ref{l:lower diagonal}.
The latter follows from the former and Lemma~\ref{l:neg-flag}.
We must now prove our finer positivity with restrictions on types of terms $y_i \ominus y_j$.

\section{Proving Theorem~\ref{t:Kmain}}
\label{s:tableaux}

Fix some compatible $\lambda, \phi$
and $\rho \subseteq \lambda$.
We say a row $i$ is a \emph{positive row}
if $\phi_i > 0$, with
\emph{negative rows}, \emph{zero rows}, 
\emph{non-positive rows} and \emph{non-negative rows} defined similarly. 
For $\nu \subseteq \lambda$, 
define $\xi(\nu)$ as a flag for $\nu'$ 
by $\xi(\nu)_i := -\phi^-_{\nu'_i}$.
Since $\phi^-$ is non-positive,
we know $\xi(\nu)$ is non-negative.
Moreover, Corollary~\ref{L: negation flag}
says $j^{\nu,  \phi^-}_\rho(\beta;\bfy)
= \omega_1(j^{\nu',  \xi(\nu)}_{\rho'}(\beta;\bfy))$.
We deduce a simple lemma regarding $\xi(\nu)$.
\begin{lem}
\label{L: conjugate 0 row}
Say $\nu/\rho$ has a cell on a non-negative row.
Then $\nu'/\rho'$ has a cell on row $j$
with $\xi(\nu)_j = 0$.
\end{lem}
\begin{proof}
Find $i$ such that $(i, \nu_i)$ lies in $\nu/\rho$ and 
$\phi_i \geq 0$.
We deduce $(\nu_i, i)$ lies in $\nu'/\rho'$
and $\phi^-_i = 0$.
It remains to show $\xi(\nu)_{\nu_i} = 0$.
Notice that $\nu'_{\nu_i} \geq i$,
so 
$$
\xi(\nu)_{\nu_i} = -\phi^-_{\nu'_{\nu_i}} \leq -\phi^-_{i} = 0.
$$
We must have $\xi_{\nu_i} = 0$ 
since $\xi(\nu)$ is non-negative. 
\end{proof}

We then describe some situations when 
$j^{\lambda, \phi^+}_\nu(\beta; \bfy)$
or $j^{\nu, \phi^-}_\rho(\beta; \bfy)$
vanish. 

\begin{lem}
\label{L: General vanish}
The polynomial $j^{\lambda,  \phi^+}_\nu(\beta;\bfy)$
vanishes if $\lambda/\nu$ has a cell on the 
diagonal or a non-positive row.
The polynomial $j^{\nu,  \phi^-}_\rho(\beta;\bfy)$
vanishes if $\nu/\rho$ has a cell on the 
diagonal or on a non-negative row.
\end{lem}
\begin{proof}
Recall~\eqref{eq:j-pos}, which says
\[
j^{\lambda, \phi^+}_\nu(\beta; \bfy)
= \sum_{\mu \subseteq \nu:\ \nu/\mu \ \mathrm{disconnected}} \beta^{|\nu| - |\mu|} G^{\phi^+}_{\lambda/\mu}(\beta; \bfx_+; \bfy) \mid_{\bfx \mapsto \bfy}.
\]

If $\lambda/\nu$ has a cell
on the diagonal
or a non-positive row,
so does $\lambda/\mu$
for every $\mu \subseteq \nu$.
Therefore, 
by Remark~\ref{R: Pos summary},
we have the first statement.

For the second statement, 
recall Corollary~\ref{L: negation flag}
says $j^{\nu,  \phi^-}_\rho(\beta;\bfy)
= \omega_1(j^{\nu',  \xi(\nu)}_{\rho'}(\beta;\bfy))$.
If $\nu/\rho$ has a cell on the diagonal,
so does $\nu'/\rho'$.
If $\nu/\rho$ has a cell on a non-negative row,
by Lemma~\ref{L: conjugate 0 row},
$\nu'/\rho'$ has a cell on row $j$ with $\xi(\nu)^-_j = 0$.
Since $\xi(\nu)$ is non-negative, 
similar to the previous paragraph, 
we have $j^{\nu',  \xi(\nu)}_{\rho'}(\beta;\bfy) = 0$
in either case.
Thus, 
$j^{\nu,  \phi^-}_\rho(\beta;\bfy)$ vanishes.
\end{proof}

Consequently, 
we deduce there is only one non-zero summand
in~\eqref{eq: general j}.

\begin{cor}
\label{C: Unique nu}
For $\rho \subseteq \lambda$,
$j^{\lambda, \phi}_\rho(\beta;\bfy)$
vanishes if $\lambda/\rho$ has a cell 
on the diagonal or on row $i$
for some $i$ such that $\phi_i = 0$.
Otherwise, 
$$
j^{\lambda, \phi}_\rho(\beta;\bfy)
= j^{\nu, \phi^-}_\rho(\beta;\bfy)j^{\lambda, \phi^+}_\nu(\beta;\bfy)
$$
where 
$$
\nu_i :=
\begin{cases}
\lambda_i \textrm{ if $\phi_i < 0$,}\\
\rho_i = \lambda_i \textrm{ if $\phi_i = 0$,}\\
\rho_i \textrm{ if $\phi_i > 0$.}
\end{cases}
$$
\end{cor}
\begin{proof}
We consider the summands in~\eqref{eq: general j}.
Suppose $\lambda/\rho$ has a cell 
on the diagonal or on a zero row.
Then for each $\rho \subseteq \nu \subseteq \lambda$,
this cell is either in $\nu/\rho$ or $\lambda/\nu$,
neutralizing either
$j^{\nu, \phi^-}_\rho(\beta;\bfy)$ 
or $j^{\lambda, \phi^+}_\nu(\beta;\bfy)$
by Lemma~\ref{L: General vanish}.

Now suppose $\lambda/\rho$ has no such cells. 
We have $\lambda_i = \rho_i$ if $\phi_i = 0$.
By Lemma~\ref{L: General vanish}
to make $j^{\lambda, \phi^+}_\nu(\beta;\bfy)$
non-zero, 
we must have $\nu_i = \lambda_i$
whenever $\phi_i < 0$.
To make $j^{\nu, \phi^-}_\rho(\beta;\bfy)$ 
non-zero, 
we must have $\nu_i = \rho_i$
whenever $\phi_i > 0$.
These conditions
uniquely determine a partition $\nu$.
\end{proof}

\begin{ex}
Suppose $\lambda = (7,4,2,2,1)$ and 
$\phi = (\overline{1}, 0, 1, 2, 4)$.
To make $j^{\lambda, \phi}_\rho$ non-zero,
we know $\lambda/\rho$ cannot have any cells
on the diagonal or in row $2$.
Suppose $\mu = (5,4,2,1,1)$.
Then by Corollary~\ref{C: Unique nu},
$j^{\lambda, \phi}_\rho(\beta;\bfy)
= j^{\nu, \phi^-}_\rho(\beta;\bfy)j^{\lambda, \phi^+}_\nu(\beta;\bfy)$,
where $\nu = (7,4,2,1,1)$.
If we change $\phi$ into
$(\overline{2}, \overline{1}, 1, 2, 3)$
and change $\mu$ into $(4,2,2,1)$,
then the $\nu$ in Corollary~\ref{C: Unique nu}
becomes $(7,4,2,1)$.
\end{ex}

Finally, for fixed $\nu$ we derive the finer version of 
$\beta$--Graham positivity.
Assume $j^{\lambda, \phi}_\rho \neq 0$.
Let $q$ be the largest such that $(q,q)$ is a cell 
in $\lambda$.
By Corollary~\ref{C: Unique nu},
we know $(q,q)$ is also a cell in $\rho$.
Thus, cells above the diagonal in $\lambda/\rho$
are above row $q$
whiles cells under the diagonal in $\lambda/\rho$
are under row $q$.
Following Remark~\ref{R: Pos summary},
we may characterize $j^{\lambda, \phi^+}_\nu(\beta; y)$.
\begin{lem}
\label{L: j+}
If we ignore the $\beta$,
we may describe $j^+ = j^{\lambda, \phi^+}_\nu(\beta; y)$
as follows:
\begin{itemize}
\item If $\phi_q > 0$,
then $j^+$ is a sum where each summand
is a distinct product of Type $1$ or Type $3$ terms.

\item If $\phi_q \leq 0$,
then $j^+$ is a sum where each summand
is a distinct product of Type $3$ terms.

\end{itemize}
\end{lem}
\begin{proof}
Recall $$j^+
= \sum_{\mu \subseteq \nu:\ \nu/\mu \ \mathrm{disconnected}} \beta^{|\nu| - |\mu|} G^{\phi^+}_{\lambda/\mu}(\beta; \bfx_+; \bfy) \mid_{\bfx \mapsto \bfy}.$$

By Remark~\ref{R: Pos summary},
each $G^{\phi^+}_{\lambda/\mu}(\beta; \bfx_+; \bfy) \mid_{\bfx \mapsto \bfy}$
is a sum where each summand
is a product of distinct terms of 
Type $1$ and $3$. 
So is $j^+$.

Now we further assume $\phi_q \leq 0$,
so $\phi^+_i = 0$ if $i \leq q$.
If $\lambda/\mu$ has a cell above the diagonal,
then that cell is on a non-positive row,
making $G^{\phi^+}_{\lambda/\mu}(\beta; \bfx_+; \bfy)$ vanish.
Thus, to make the summand non-zero, 
$\lambda/\mu$ must be under the diagonal.
By (Case Down) of Remark~\ref{R: Pos summary},
$j^+$
is a sum of products of distinct Type $3$ terms.
\end{proof}

We can deduce a similar
statement of $j^{\nu, \phi^-}_\rho(\beta; y)$.
\begin{lem}
\label{L: j-}
If we ignore the $\beta$,
we may describe $j^- = j^{\nu, \phi^-}_\rho(\beta; y)$
as follows:
\begin{itemize}
\item If $\phi_q < 0$,
then $j^-$ is a sum where each summand
is a distinct product of Type $2$ or Type $3$ terms.
\item If $\phi_q \geq 0$,
then $j^-$ is a sum where each summand
is a distinct product of Type $3$ terms.
\end{itemize}
\end{lem}
\begin{proof}
We define flag $\xi$ for $\nu'$ 
by $\xi_i := -\phi^-_{\nu'_i}$,
so $\xi$ has only non-negative entries.

Recall $$j^-
= \omega(j^{\nu', \xi(\nu)}_{\rho'})
= \omega \left(\sum_{\mu \subseteq \rho:\ \rho/\mu \ \mathrm{disconnected}} \beta^{|\rho| - |\mu|} G^{\xi(\nu)}_{\nu'/\mu'}(\beta; \bfx_+; \bfy) \mid_{\bfx \mapsto \bfy}\right).$$

By Remark~\ref{R: Pos summary},
each $G^{\xi(\nu)}_{\nu'/\mu'}(\beta; \bfx_+; \bfy) \mid_{\bfx \mapsto \bfy}$
is a sum where each summand
is a product of distinct terms of 
Type $1$ and $3$. 
Under $\omega$,
Type $1$ terms become Type $2$
while Type $3$ terms remain
Type $3$.
Thus, 
$j^-$ is a sum where each summand
is a product of distinct terms of 
Type $2$ and $3$. 

Now we further assume $\phi_q \geq 0$.
If $\nu/\mu$ has a cell under the diagonal, 
say in row $i$,
we know $i \geq q$.
Thus, that cell is on a non-negative row.
By Lemma~\ref{L: conjugate 0 row},
$\nu'/\mu'$ has a cell in row $j$
with $\xi(\nu)_j = 0$,
making $G^{\xi(\nu)}_{\nu'/\mu'}(\beta; \bfx_+; \bfy)$ vanish.
To make the summand non-zero, 
$\nu/\mu$ must be above the diagonal,
so $\nu'/\mu'$ is under the diagonal.
By (Case Down) of Remark~\ref{R: Pos summary},
$j^{\nu', \xi(\nu)}_{\rho'}$
is a sum of products of distinct Type $3$ terms.
So is $j^-$.
\end{proof}

Finally, we combine 
Lemma~\ref{L: j+}
and Lemma~\ref{L: j-}
to obtain the finer posivity
of $j^{\lambda, \phi}_\rho = j^+ j^-$.

\begin{thm}
\label{t:Krefined-ext}
Either (or both) of the following two cases hold.
\begin{itemize}
\item Case 1: $\phi_q \leq 0$.
Here $j^{\lambda, \phi}_\rho(\beta;\bfy)$
is a sum of products
of Type 2 and Type 3 terms with a power of $\beta$.
In each summand, 
every Type 2 (resp. 3) term appears at most once
(resp. twice). 
\item Case 2: $\phi_q \geq 0$.
Here $j^{\lambda, \phi}_\rho(\beta;\bfy)$
is a sum of products
of Type 1 and Type 3 terms with a power of $\beta$.
In each summand, 
every Type 1 (resp. 3) term appears at most once
(resp. twice). 
\end{itemize}
\end{thm}
\begin{proof}
Immediate from Lemma~\ref{L: j+}
and Lemma~\ref{L: j-}.
\end{proof}

\section{Final Remarks}
\label{s:remarks}

We conclude with some brief remarks, including suggestions for further directions.

\subsection{An explicit tableaux formula}
Recall for compatible $\lambda$ and $\phi$ that $w_{\lambda,\phi}$ is the unique vexillary permutation with shape $\lambda$ and flag $\phi$. 
The permutation $\neg w_{\lambda, \phi}$ is also vexillary, e.g., by Lemma~\ref{l:neg-flag}. 

We extend the map $\omega_1$ acting on polynomials to one acting on tableaux.
A set-valued tableau $T$ is 
\emph{row-strict-decreasing}
if entries of $T$ strictly decrease in each row and weakly decrease in each column. 
The \emph{$r$--weight} of an element in cell $(r,c)$ of the row-strict-decreasing tableau $T$ is $y_{i + c - r} \ominus x_i$.
Then the \emph{$r$--weight} of $T$, denoted $\wt_r(T)$, is the product of each element weight.

For $\lambda$ a partition,  sequences of the form $\phi = (\phi_1 \geq \dots \geq \phi_{\lambda_1})$ are \emph{reverse flags}
for $\lambda$. 
Let $\RSDT^\phi(\lambda/\mu)$ 
be the set of row-strict-decreasing tableaux
with shape $\lambda/\mu$
such that entries in column $i$ are weakly
larger than $\phi_i$.
Recall $\RSDT_-^\phi(\lambda/\mu)$ is the subset of $\RSDT^\phi(\lambda/\mu)$
consisting of tableaux whose entries are subsets of $\ZZ_\leq$. 

Let $\omega_1$ act on set--valued tableaux by conjugating, then replacing each value $i$ with $1-i$.
To justify our use of the notation $\omega_1$, observe:
\begin{lem}
\label{L: Omega 1 on tableaux}
The map $\omega_1:\SetSSYT(\lambda/\mu) \to \RSDT(\lambda'/\mu')$ is a bijection
with 
\[
\wt_r(\omega_1(T)) = \omega_1(\wt(T)).
\]
Moreover, 
$\omega_1$ restricts to a bijection from 
$\SetSSYT^\phi_{+}(\lambda/\mu)$ to $\RSDT_-^{1 - \phi}(\lambda'/\mu')$
where $1 - \phi$ is the reverse flag obtained by replacing
every $i$ in $\phi$ by $1 - i$.
\end{lem}

As a consequence, we see $a^{\lambda,\phi_-}_\rho(\beta;\bfx;\bfy)$ can be expressed as a sum over row strict decreasing tableaux with weight function $\wt_r$.
Then~\eqref{eq: general a} can be viewed as a sum over tableaux, with positive entries forming a semistandard tableau weighted according to $\wt$ and non-positive entries forming a row strict decreasing tableau weighted according to $\wt_r$.

\subsection{More permutations} A complete combinatorial proof that double $\beta$--Edelman--Greene coefficients are $\beta$--Graham positive remains an open problem, even for $\beta =0$.
Our proof depends critically on having a tableau formula for $\cev{\fkG}_w(\beta;\bfx;\bfy)$, which only exists for $w$ vexillary.
To gain a tableau formula for all $w \in S_\ZZ$, one could extend the Lascoux decomposition of the Grothendieck polynomial $\fkG_w(\beta;\bfx_+)$ to $\cev{\fkG}_w(\beta;\bfx;\bfy)$.
We have explored a candidate for this expansion in the cohomology case.
In principle extending our approach appears plausible, but any extension of Lemma~\ref{l:flag-symm} would be far more subtle.

A relatively easy extension of Theorem~\ref{t:Kmain} relies on a recurrence for $\beta$--Grothendieck polynomials called the transition equations~\cite{lascoux2001transition,weigandt2021bumpless}, instances of which are determined by picking certain boxes from the Rothe diagram.
If we can apply transition to a box $(r,c)$ with $r \prec c$ for $w \in S_\ZZ$ and the resulting terms in the recurrence have Graham positive double $\beta$--Edelman--Greene coefficients, so will $w$.
Transition equations terminate with vexillary permutations, and this approach yields many more $w \in S_\ZZ$ for which the conclusion of Theorem~\ref{t:Kmain} applies, but this set is not easily characterized.

\subsection{Non-vanishing coefficients}
Recall the Edelman--Greene coefficient of $\mu$ for $w \in S_\ZZ$ is $j^w_\mu := j^w_\mu(0)$. These are the codimension zero double Edelman--Greene coefficients.
Note that $\lambda(w)$ is well-defined even for non-vexillary permutations, and let $\ll$ denote dominance order for integer partitions.
Billey and Pawlowski have given an elegant necessary condition for Edelman--Greene coefficients to be non-zero.

\begin{lem}[{\cite[Lem 3.11]{billey2014permutation}}]
\label{l:billey-pawlowski}
For $w \in S_\ZZ$, $j^w_\mu = 0$ unless $\lambda(w) \ll \mu \ll \lambda(w^{-1})'$.
    
\end{lem}

From this perspective, $w$ is vexillary if and only $\lambda(w) = \lambda(w^{-1})'$.
It is natural to ask:
\begin{ques}
    For $w \in S_\ZZ$ and $\mu$ a partition, when is $j^w_\mu(\beta;\bfy)$ non-zero?
\end{ques}
This would be an interesting question to understand even when  setting $\beta = 0$.
In the vexillary case for $\mu \subseteq \lambda(w)$ so that $\lambda(w)/\mu$ contains no entries on the diagonal, our tableau formula allows us to determine when whether the coefficient vanishes in polynomial time.
Otherwise the coefficient vanishes.
Additionally, we propose extending Lemma~\ref{l:billey-pawlowski} to double $\beta$--Edelman--Greene coefficients as an interesting problem, and further motivation for understanding such coefficients in general.
For this problem, both the $\beta = 0$ and $\bfy \mapsto 0$ cases are also of interest.

\subsection{Variants}
One promising avenue is to define backstable Schubert polynomials and double Stanley symmetric functions in the other types.
In many ways, the double Schubert polynomials of~\cite{ikeda2011double} and double $\beta$--Grothendieck polynomials~\cite{kirillov2017construction} in Types B, C and D are \emph{already} backstable\footnote{This observation was first made in~\cite{marberg2021principal}, where the single backstable $\beta$--Grothendieck polynomial was first defined.} with formulas akin to Proposition~\ref{p:Kbackstable-expansion}, so such a definition is feasible.
As evidence for the existence of these objects and corresponding Graham positivity, one could attempt to extend Theorems~\ref{t:Kmain} to flagged double Schur-P and-Q functions and their $K$--theoretic analogues.
Flagged double Schur-P and-Q functions are introduced via Pfaffian formulas in ~\cite{anderson2012degeneracy,anderson2018chern} with a tableau formula for the latter in~\cite{matsumura2023tableau}.

\subsection{Schubert structure coefficients}
One of the most important open problems in algebraic combinatorics is to compute Schubert structure coefficients.
The most classical problem of this form is solved by the Littlewood--Richardson rule:
\[
s_\lambda \cdot s_\mu = \sum_{\nu} c^\nu_{\lambda\mu} s_\nu,
\]
where $c^\nu_{\lambda\mu}$ has a combinatorial description.
This rule has been generalized to the equivariant setting~\cite{sagan2001symmetric,knutson2003puzzles}, to $K$--theory~\cite{buch2002littlewood} and to both simultaneously~\cite{pechenik2017equivariant}.
A combinatorial rule for the product of Schubert polynomials
\begin{equation}
    \label{eq:structure}
    \fkS_u(\bfx) \cdot \fkS_v(\bfx) = \sum_w c^w_{uv} \fkS_w
\end{equation}
is open.
Similarly, there are generalizations to the equivariant setting where $c^w_{uv}$ is a Graham positive polynomial, and to $K$--theory, as well as to both simultaneously.
Indeed, Graham positivity was introduced to understand the equivariant setting of this problem~\cite{graham2001positivity}.
As mentioned, the coefficients for single Stanley symmetric functions and their $K$--theoretic analogues have been used to prove rules for special cases of~\eqref{eq:structure}.
This raises:

\begin{ques}
    Can equivariant ($K$--theoretic) Schubert structure coefficients be computed using double ($\beta$--)Edelman--Greene coefficients?
\end{ques}

It would be an accomplishment to recover equivariant Grassmannian cohomology~\cite{knutson2003puzzles} or $K$--theory structure coefficients~\cite{pechenik2017equivariant} in this way.
One piece of evidence that this question could prove fruitful is that the proof from~\cite{anderson2023strong} that the $\cev{\fkG}_w(\beta;\bfx;\bfy)$'s are closed under multiplication relies on the Graham positivity of double Stanley symmetric functions.
Additionally, Thomas Lam has informed us that double $\beta$--Edelman--Greene coefficients should compute equivariant $K$--homology for the Grassmannian, first computed in~\cite{wheeler2019littlewood}.
Additionally, we suspect these coefficients should generalize results from~\cite{molev2009co} to double Grothendiecks.

\

\noindent \textbf{Acknowledgements:} We thank Dave Anderson for helping us better understand the geometric setting for our work and Thomas Lam for sharing his proof outline for Proposition~\ref{p:backstable-grothendieck-ring}.

\appendix

\section{Details on double symmetric functions and $K$--theoretic extensions}
\label{a:functions}

The main goal of this Appendix is to prove Proposition~\ref{p:backstable-grothendieck-ring}.
We begin by remarking that, while~\cite{lam2021Ktheory} works with the specialization $\beta = -1$, by inserting powers of $\beta$ so that all terms are homogeneous with the convention $\deg \beta = -1 $, we recover our conventions.
This is explained in full detail in~\cite[Rem~7.2]{anderson2023strong}.

Let $p_k(\bfx_-;\bfy) = p_k(\bfx_-/\bfy) = \sum_{i \leq 0} x_i^k - y_i^k$, and define $\Lambda(\bfx_-;\bfy)$ as the $\CC[\bfy]$--algebra generated by the $p_k(\bfx;\bfy)$'s for all $k$.
For any symmetric function $f(\bfx_-)$, we can expand $f$ into the power sum basis and apply the map $p_k(\bfx_-) \mapsto p_k(\bfx_-/\bfy)$ to construct a new symmetric function $f(\bfx_-/\bfy)$.
For $M$ a graded algebra, let $M^\beta$ be the \emph{homogenized graded completion} of $M$, that is the set of formal power series
\[
\left\{\sum_{i = k}^\infty \beta^{i - k}m_i: k \in \ZZ_+, \deg(m_i) = i \ \mbox{in}\ M\right\}
\]
with $\deg(\beta) = -1$.
Then $\Gamma(\beta;\bfx_-;\bfy) = \mbox{span}_{R(\beta;\bfy)}(\{G_\lambda(\beta;\bfx_-;\bfy)\})$ is a subset of $\Lambda^\beta(\bfx_-;\bfy)$.
In~\cite{lam2021Ktheory}, it is shown for $w \in S_\ZZ$ that $\cev{\fkG}_w(\beta;\bfx;\bfy) \in R(\beta;\bfx;\bfy) \otimes \Lambda^\beta(\bfx_-;\bfy)$.
Our goal is to prove:
\begin{prop}[{Proposition~\ref{p:backstable-grothendieck-ring}}]
    For $w \in S_\ZZ$, $\cev{\fkG}_w(\beta;\bfx;\bfy)$ lies in $R(\beta;\bfx;\bfy) \otimes \Gamma(\beta;\bfx_-;\bfy)$.
\end{prop}

Recall the \emph{Grothendieck polynomial} of $w \in S_\infty$ is $\fkG_w(\beta;\bfx_+) = \fkG_w(\beta;\bfx_+;\bfy_+)\mid_{\bfy \mapsto 0}$.
The (single) \emph{stable Grothendieck polynomial} of $\lambda$ is $G_\lambda(\beta;\bfx_-) = G_\lambda(\beta;\bfx_-;\bfy) \mid_{\bfy \mapsto 0}$.
The (single) \emph{$\beta$--Stanley symmetric function}\footnote{More commonly referred to as the stable Grothendieck polynomial of $w$.} of $w \in S_\ZZ$ is $F_w(\beta;\bfx_-) = F_w(\beta;\bfx_-;\bfy)\mid_{\bfy \mapsto 0}$.
As shown in~\cite{buch2008stable}
\[
F_w(\beta;\bfx_-) \in \ZZ_{\geq 0}[\beta^{m - \ell(w)}G_\lambda(\beta;\bfx_-):\lambda \vdash m \geq \ell(w)].
\]

Define the $0$--Hecke monoid $(S_\ZZ,*)$ where $u * v = w$ when $\pi_u *\pi_v = \beta^{\ell(u) + \ell(v) - \ell(w)}\pi_w$.
Let $S_{-\infty} = \omega(S_\infty)$, and for $w \in S_{-\infty}$ define $\fkG_w(\beta;\bfx) = \tilde{\omega}(\fkG_{\neg w}(\beta;\bfx))$.
Let $S_{\neq 0} = S_{-\infty} \times S_\infty$, and for $w = u \times v \in S_{\neq 0}$ with $u \in S_{-\infty}$, $v \in S_\infty$ define $\fkG_w(\beta;\bfx) = \fkG_u(\beta;\bfx) \cdot \fkG_v(\beta;\bfx)$.
Note $S_{\neq 0}$ is generated by $\{s_i: i\in \ZZ - \{0\}\}$.

\begin{prop}[{\cite[Prop~5.17]{lam2021Ktheory}}]
    \label{p:Kbackstable-expansion}
    For $w \in S_\ZZ$,
    \[
    \cev{\fkG}_w(\beta;\bfx;\bfy) = \sum_{\substack{
    u * v * z = w\\
    u,z \in S_{\neq 0}
    }} \fkG_{u^-1}(\ominus\bfy) \cdot F_v(\beta;\bfx_-/\bfy) \cdot \fkG_z(\beta;\bfx).
    \]
\end{prop}

\begin{proof}[Proof of Proposition~\ref{p:backstable-grothendieck-ring}]
By Propositions~\ref{p:Kgrassmannian} and~\ref{p:Kbackstable-expansion}, we have
\[
G_\lambda(\beta;\bfx;\bfy) = \cev{\fkG}_{w_\lambda}(\beta;\bfx;\bfy) = \sum_{\substack{u * v = w_\lambda\\ u \in S_{\neq 0}}} \fkG_{u^{-1}}(\beta;\ominus \bfy) \cdot F_v(\beta;\bfx_-/\bfy).
\]
Here the $\fkG_z(\beta;\bfx)$ term is omitted since every reduced word for $w_\lambda$ ends in $s_0$, hence any $z$ from Proposition~\ref{p:Kbackstable-expansion} must be the identity, and $\fkG_1(\beta;\bfx) = 1$.
This equation implies $G_\lambda(\beta;\bfx;\bfy) \in \Lambda(\beta;\bfx;\bfy)$.
To complete our proof, move all terms with $v \neq w_\lambda$ to the left hand side, giving
\[
F_\lambda(\beta;\bfx_-/\bfy) \cdot \sum_{\substack{u*w_\lambda = w_\lambda\\ u \in S_{\neq 0}}} \fkG_{u^{-1}}(\beta;\ominus \bfy) = G_\lambda(\beta;\bfx_-;\bfy) - \sum_{\substack{u*v = w_\lambda\\ u \in S_{\neq 0},\ v \neq w_\lambda}} \fkG_{u^{-1}}(\beta;\ominus \bfy) F_v(\beta;\bfx_-/\bfy).
\]
Note $v$ is also $0$--Grassmannian with $\ell(v) < |\lambda|$, so by induction on $|\lambda|$, we can assume the right hand side belongs to $\Gamma(\beta;\bfx;\bfy)$.
The sum in the left side is finite and $\fkG_w(\beta;\ominus \bfy) \in R(\beta;\bfy)$ for all $w \in S_{\neq 0}$, so the coefficient of $G_\lambda(\beta;\bfx_-/\bfy)$ is in $R(\beta;\bfy)$ as desired.
\end{proof}

\printbibliography

\end{document}